\documentclass[10pt]{amsart}
\usepackage{etex}
\usepackage[french, english]{babel}
\usepackage{tikz}
\usetikzlibrary{matrix,arrows}
\usepackage[all,cmtip]{xy}
\usepackage{enumerate}
\usepackage{dsfont, baskervald}
\usepackage{ marvosym }
\usepackage{mathrsfs}
\usepackage{todonotes}
\usepackage{setspace}
\usepackage{framed}
\usepackage{tabularx, tabu}
\usepackage{longtable}
\usepackage{tikz,tikz-cd}
\usepackage{amsfonts,amsthm,amsmath,amssymb,comment,stmaryrd,mathalfa,url,mathrsfs, yfonts}
\usepackage{graphicx}
\graphicspath{ {images/} }
\usepackage[section]{placeins}
\usepackage{microtype}
\usepackage{booktabs, calc, multirow}
\usepackage{array}
\usepackage[ocgcolorlinks,backref=page]{hyperref}
\hypersetup{citecolor=blue}
\usepackage{amscd, fix-cm}
\usepackage{makecell}
\usepackage{chngcntr}
\graphicspath{ {images/} }

\usetikzlibrary{%
	matrix,%
	calc,%
	arrows%
}

\def\AA{\mathbb{A}}
\def\CC{\mathbb{C}}

\def\QQ{\mathbb{Q}}

\def\ZZ{\mathbb{Z}}
\def \lra{\longrightarrow}

\def \a{\alpha}

\def \z {\zeta}

\def \A {\mathcal{A}}
\def \g {\gamma}
\def \w {\omega}
\def \LL {\mathcal{L}}
\def \W {\mathcal{W}}

\def \S {\mathcal{S}}
\def \s {\sigma}

\def \OO {\mathcal{O}}

\def \s {\sigma}

\newcommand{\matrd}[4]{\begin{pmatrix}#1&#2\\#3&#4\end{pmatrix}}
\newcommand{\smallmatrd}[4]{\left(\begin{smallmatrix}#1&#2\\#3&#4\end{smallmatrix}\right)}

\def \un {\underline}
\def \ov{\overline}

\def \tors {\text{tors}}

\def \wh{\widehat}

\def\JJ{\mathfrak{J}}

\def \nn {\mathfrak{n}}

\def \k {\kappa}
\def \Norm {N_{F/\QQ}}

\def\ps{\mathfrak{p}}

\DeclareMathOperator{\Sym}{Sym}

\DeclareMathOperator{\val}{val}

\DeclareMathOperator{\NP}{NP}

\DeclareMathOperator{\Gal}{Gal}

\theoremstyle{plain}
\newtheorem{thm}{Theorem}[section]
\newtheorem*{thm*}{Theorem}
\newtheorem{lem}[thm]{Lemma}
\newtheorem{cor}[thm]{Corollary}

\newtheorem*{con*}{Conjecture}
\newtheorem{prop}[thm]{Proposition}
\newtheorem*{prop*}{Proposition}
\theoremstyle{definition}
\newtheorem{defn}[thm]{Definition}
\newtheorem{exmp}[thm]{Example}
\theoremstyle{definition}
\newtheorem{rmrk}[thm]{Remark}
\newtheorem*{rmrk*}{Remark}

\newtheorem*{exmp*}{Example}

\newtheorem*{obs*}{Observation}
\newtheorem{nota}[thm]{Notation}

\usepackage[margin=1 in]{geometry}

\numberwithin{equation}{section}

\title{$2$-adic slopes of Hilbert modular forms over $\QQ(\sqrt{5})$}
	
\author{Christopher Birkbeck}	
\address{\newline Department of Mathematics, University College London, Gower Street, London, WC1E 6BT \newline Email: c.birkbeck@ucl.ac.uk \newline ORCID ID: orcid.org/0000-0002-7546-9028 }

\begin{document}

\maketitle

\begin{abstract}
	We show that for arithmetic weights with a fixed finite order character, the slopes of $U_p$ for $p=2$ (which is inert) acting on overconvergent Hilbert modular forms of level $U_0(4)$ are independent of the (algebraic part of the) weight and can be obtained by a simple recipe from the classical slopes in parallel weight $3$. 
\end{abstract}

\section{Introduction}

Modular forms and more generally overconvergent modular forms of level divisible by $p$ (a prime number) are acted upon by a Hecke operator denoted $U_p$. This is a compact operator on the spaces of overconvergent modular forms and therefore it makes sense to study its eigenvalues. One of the aspects one can study is the $p$-adic valuation of the eigenvalues; these are known as the slopes, since they can be computed from the Newton polygon of the characteristic power series of $U_p$. One of the aspects highlighted by the work of Gouvea--Mazur and Hida (among others) was the degree to which the slopes depend on the weight of the space of modular forms. For classical elliptic modular forms a weight is simply a positive integer, but for overconvergent modular forms one can allow more general weights given by continuous homomorphisms $\kappa:\ZZ_p^\times \to \CC_p^\times$. To these weights one can associate a rigid analytic space $\W$ known as the weight space which one can check is simply a finite union of open discs. The weight space can be divided into two regions as follows: let $q=p$ if $p$ is odd and $q=4$ if $p=2$ and let $\gamma$ be a topological generator of $1+q\ZZ_p$. For $\kappa$ a weight we define $w(\k)=\k(\gamma)-1$. Then we say $\k$ is in the centre of weight space if $\val_p(\w(\k)) \geq 1$ for $p$ odd and $\geq 3$ for $p=2$. We say a weight is near the boundary if $\val_p(\w(\k)) \leq \frac{1}{p-1}$ for $p$ odd and $\val_p(\w(\k)) < 3$ for $p=2$.

 Buzzard and Kilford \cite{BK} studied slopes of overconvergent (elliptic) modular forms in the special case that the tame level is $1$ and $p=2$. In this case they proved that, for weights near the boundary of weight space, slopes have a great deal of structure, in particular they are in arithmetic progression and scale to zero as the weight moves further towards the boundary of weight space. From this one can deduce that, over the boundary annulus, the eigencurve is a disjoint union of annuli. This example motivated further work in the area (see \cite{roe,slbd}) and most notably Liu--Wan--Xiao in \cite{eovb} showed that the slopes of quaternionic modular forms (over $\QQ$) have structure analogous to what was observed by Buzzard--Kilford. Specifically, they showed that the slopes (for weights near the boundary) are given as a union of arithmetic progressions with common difference and they scale as the weight approaches the boundary. From this  they then deduce that, over the boundary annulus, the eigenvarieties are a disjoint union of rigid spaces which are flat over this annulus. Then, via Chenevier's overconvergent Jacquet--Langlands correspondence, one deduces similar results for a large class of spaces of overconvergent modular forms. This in turn can be used to prove parity conjectures of Selmer ranks for modular forms.


In the setting of Hilbert modular forms, much less is known. In this case, weights have several components and therefore the boundary is more complicated as one can approach it via each of the components separately. Computations in \cite{ME} suggest that, again for weights near the boundary (see \cite[Definition 2.1.8]{ME} for a precise definition), slopes behave in a highly predictable way but in general are not a union of arithmetic progressions with common difference. In this more general setting one observes that the multiset of slopes appears to be completely determined by a simple recipe starting with slopes of classical Hilbert modular forms of some small weight. Similar behaviour is also present for elliptic modular forms, yet due to their simpler nature, in this case, the above recipe coincidentally gives unions of arithmetic progressions with common difference. This combined with the scaling behaviour of the slopes gives the above statements about the geometry of the associated eigenvarieties. 

In the Hilbert setting, one can also consider the effect of the splitting behaviour of $p$.  If $p=\prod_{i} \ps_i$ then $U_p=\prod_i U_{\ps_i}$ and by making certain restrictions on the weight space\footnote{Specifically, fixing certain components of the weights.} one can study the slopes of $U_{\ps_i}$. If $p$ is totally split, work of Newton--Johansson \cite{NewtJoh}, shows that the methods of \cite{eovb} can be used to describe the slopes of $U_{\ps_i}$, which for general $p$ is not the case.\footnote{This is due to the fact that the associated Newton and Hodge polygons in general only touch at the base point (see \ref{rmrk:HN}).} Moreover, they construct partial eigenvarieties and prove that over a boundary annulus these partial eigenvarieties decompose as a union of components which are finite over weight space. From this they prove the parity part of the Bloch--Kato conjecture for Galois representations associated to Hilbert modular forms (still with $p$ totally split). 

In general, the structure of the slopes of $U_p$ in the Hilbert case is unknown. In this note, we show that for an inert prime ($p=2$) and for arithmetic weights with a \textit{fixed} finite order character, the slopes of $U_2$ acting on spaces of overconvergent Hilbert modular forms of level $U_0(4)$ (over $\QQ(\sqrt{5})$) are completely determined by the slopes of classical Hilbert modular forms of parallel weight $3$ with a suitable finite order character. We also give computational evidence that as weights approach the boundary of weight space, the slopes scale and tend to zero (similar to what is seen for elliptic modular forms).  Specifically, we prove the following result:

\begin{thm*}
	The slopes of\/ $U_2$ acting in weight $[n_1,n_2]\chi$ on overconvergent Hilbert modular forms of level $U_0(4)$ is independent of $n_i$. Moreover, they are completely determined by the slopes of\/ $U_2$ acting on the classical spaces of Hilbert modular forms of level $U_0(4)$, parallel weight $3$ with nebentypus $\chi \tau^n$ for a $\tau$ a certain character of order $6$. 
\end{thm*}

This theorem is in line with  \cite[Conjecture 4.7.9]{ME}, appropriately generalized to this situation where the level is not \textit{sufficiently small}. The  proof is essentially an extended exercise on $p$-adic matrix analysis which we do not know how to generalize to many other situations. Specifically, it relies on the fact that, for this example, knowledge of the $p$-adic valuations of the entries of the matrix associated to $U_p$ completely determine the slopes of its eigenvalues, which is not always the case.\footnote{As one can check in the examples computed in \cite{ME}.} Yet we note that the methods used here are of a different nature than those in \cite{eovb}, which rely on studying the associated Hodge polygons, which in this example give only trivial bounds on the associated Newton polygon.

The above theorem together with the computational evidence suggest that, in this example, the slopes tend to zero as one approaches the boundary\footnote{In these computations we approach the boundary  in all components of the weight space at once.} of weight space. In particular, this would imply, via similar methods to those in \cite{NewtJoh}, that over the boundary annulus the corresponding eigenvariety is a disjoint union of rigid spaces which are finite flat over this boundary annulus. Moreover, we expect that methods similar to those of \cite{NewtJoh}, would also prove the parity part of the Bloch--Kato conjecture for this specific example (if the scaling behaviour of the slopes were proven).

\subsection*{Acknowledgements}
This paper owes it existence to very useful conversations with Kevin Buzzard, for which I am very grateful. I would also like to thank the anonymous referee for their useful comments.  The work was done while the author was an EPSRC Doctoral Prize Fellow at UCL.

\section{Background}

Let us begin with the relevant set-up. Throughout we will set $p=2$. We want to study the $p$-adic slopes of overconvergent Hilbert modular forms over $F=\QQ(\sqrt{5})$. We note that $2$ is inert in $F$ and that $F$ has narrow class number one which will simplify the computations.  As is usual, instead of working directly with the space of overconvergent Hilbert modular forms, we will instead work with spaces of overconvergent quaternionic modular forms as the geometry is simpler. Specifically, by \cite[Theorem 1]{me2}, if we let $D$ be the unique quaternion algebra over $F$ ramifying only at the two infinite places of $F$, then the eigenvariety associated to Hilbert modular forms (as defined by \cite{AIP})  is isomorphic to the eigenvariety associated to these quaternionic forms. Therefore, since we are only interested in slopes, there is no loss in working with overconvergent quaternionic forms. We begin by recalling their definition.

\begin{nota}Throughout, $D$ will be the totally definite quaternion algebra, which we note has class number one and we let $\OO_D$ denote a fixed maximal order. Specifically, we let $\OO_D$ be the icosian ring, since any other maximal order is conjugate to this (see \cite[Section 3]{Dem}). Moreover, we fix a splitting of $D$ at $p$, let $\wh{\OO}_D=\OO_D \otimes \wh{\ZZ}$ and set $U_0(p^s):= \left \{\g \in \wh{\OO}_D ^\times : \g \equiv \left ( \begin{smallmatrix} *&*\\ 0&*  \end{smallmatrix} \right) \bmod p^s \right \}$. 	
\end{nota}

\begin{defn}
Let $n \in \ZZ^{2}$ and $v \in \ZZ^{2}$ such that $n+2v=(r,r)$ for some $r \in \ZZ$. Set $k=n+2$ and $w=v+n+1$ (understood as the obvious component-wise sum).	 It follows from the above that all the entries of $k$ have the same parity and $k=2w-r$. Moreover,  note that given $k$ (with all entries paritious and greater than $2$) and $r$ we can recover $n,v,w$. We call the $5$-tuple $(k,r,n,v,w)$ a weight tuple and, in order to simplify notation, we denote it simply by $[n_1,n_2]$, where it is understood that $k,r,v,w$ are  implicit.
\end{defn}

\begin{defn}
	Let $K$ denote the unramified extension of $\QQ_{p}$ of degree $2$ and let $\OO_K$ denote its ring of integers. Note that $\OO_K \cong \OO_F \otimes \ZZ_{p}$.  For $\chi$ a finite character on $\OO_F$ we consider it as a character on $\OO_K$ via weak approximation. Then for an algebraic weight $[n_1,n_2]$ we define a weight-character $[n_1,n_2]\chi$ as the map $\OO_K \to \CC_p$ defined by $\a \mapsto \a^{n_1}\ov{\a}^{n_2}\chi(\a)$, where for $\a \in \OO_K$ we let $\ov{\a}$ denote the image of $\a$ under the action of the non-trivial element in $\Gal(K/\QQ_{p})$. Such weight-characters are called arithmetic weights.
\end{defn}

\begin{defn} Let $D_f^\times:=D\otimes \AA_{F,f}$. The space of overconvergent quaternionic modular forms of weight $\k=[n_1,n_2]\chi$, level $U_0(p^s)$ (and radius of overconvergence $1$),  denoted by $S_{\kappa}^{D,\dagger}(U_0(\nn p^s))$,   is defined as the vector space of functions \[f: D^\times \backslash D_f^\times  \lra K\langle X,Y \rangle\] such that $f(dg)=f(g)$ for all $d \in D^\times$ and $f(gu^{-1})\cdot_{\kappa_\psi} u_p=f(g)$ for all $u \in U_0(\nn p^s)$ and $g \in D_f^\times$. Here the action of $\g= \left(\begin{smallmatrix} a&b\\ c&d \end{smallmatrix} \right)  \in U_0(p^s)$ on $ K\langle X,Y \rangle$ is given by 
\begin{equation}
X^l Y^m \cdot_{\kappa} \g=\chi(d) \det(\g)^{v_1}\det(\ov{\g})^{v_2}(cX+d)^{n_1}(\ov{c}Y+\ov{d})^{n_2} \left( \frac{aX+b}{cX+d} \right)^l \left( \frac{\ov{a}Y+\ov{b}}{\ov{c}Y+\ov{d}} \right)^m
\end{equation}	
	 where $(k,r,n,v,w)$ is the associated weight tuple. To ease notation we are ignoring the radius of overconvergence since this will not affect our result.
\end{defn}

\begin{rmrk}
	In what follows the computations were done in MAGMA (\cite{Magma}). The specific Magma code is available at \cite{githubme}.
\end{rmrk}

 From \cite[Section 3]{Dem} we see that there is a bijection between $D^\times \backslash D_f^\times / U_0(p^2)$ and $\OO_D^\times \backslash \wh{\OO}_D^\times/ U_0(p^2)$, which in turn is in bijection with $\OO_D^\times \backslash \mathbb{P}^1(\OO_F/p^2)$. So it suffices to compute the orbit of elements in $\mathbb{P}^1(\OO_F/p^2)$ under $\OO_D^\times$. Doing this one finds that there is a single orbit which is represented by the element $[1:0] \in \mathbb{P}^1(\OO_F/p^2)$. We then lift this to an element  $t \in \wh{\OO}_D^\times$ which is trivial at all finite places different from $p$ and at $p$ is given by $\smallmatrd{0}{-1}{1}{0}$.  Therefore, $\wh{\OO}_D^\times= \OO_D^\times t U_0(p^2)$ . Now, by evaluating at $t$ we get an isomorphism  \[S_{\kappa}^{D,\dagger}(U_0(p^2)) \overset{\sim}{\lra} K\langle X,Y	 \rangle^{\Gamma^t}\] where $\Gamma^t=t D^\times t^{-1} \cap U_0(p^2)$.

\begin{nota}
 From now on we let $\z$ denote a fixed non-trivial cube root of unity in $K$.
\end{nota}
\begin{prop}\label{prop26}
	In the above set-up,  $\Gamma^t/ \Gamma^t \cap F^\times$ is cyclic of order $3$. Furthermore, at $p$,  the matrix $\left (\begin{smallmatrix}
	\zeta &0 \\ 0 & \zeta^{-1} \end{smallmatrix}
	\right )$ is a generator of this group.
\end{prop}

\begin{proof}
By \cite[Lemma 7.1]{hidaon}, this group (denoted $\overline{\Gamma}^t(U_0(p^2))$ in $loc. cit.$) is finite since $D$ is totally definite. 

Now, using the bijection between  $D^\times \backslash D_f^\times / U_0(p^2)$ and $\OO_D^\times \backslash \mathbb{P}^1(\OO_F/p^2)$ from \cite[Section 3]{Dem} one can identify this group with the stabilizer of $t$ in $\OO_{D,1}^\times / \{\pm 1\}$ where $\OO_{D,1}^\times$ are the units of norm one. This group is then easily computed as well as its image at $p$.\qedhere

\end{proof}

Now, in level $U_0(p^s)$, the action  of $U_{p}$ is given by 	
\begin{align}
(f|U_p)(t)= &\sum_{\a \in \OO_K/ p} f|_{u_\a}(t) \\ =&  \sum_{\a \in \OO_K/ p} f(t u_\a^{-1})\cdot (u_\a)_p 
\end{align} where $u_\a= \left(\begin{smallmatrix} p &0 \\ \a p^{s}& 1 \end{smallmatrix} \right)$ and $(-)_p$ denotes the $p$-part. If we write $t u_\a^{-1}=dtv_\a$ with $d \in D^\times$, $v_\a \in U_0(p^s)$ then \[(f|U_p)(t)=\sum_{\a \in \OO_K/ p} f(t)\cdot (v_\a u_\a)_p= \sum_{\a \in \OO_K/ p} f(t) \mid_{v_\a u_\a}.\] Therefore, via $S_{\kappa}^{D,\dagger}(U_0(p^2)) \overset{\sim}{\lra} K\langle X,Y	 \rangle^{\Gamma^t}$ the action of $U_p$ on $ K\langle X,Y	 \rangle^{\Gamma^t}$ is given by $\sum_{\a \in \OO_K/ p} \mid_{v_\a u_\a}$.

\begin{prop}\label{prop1}
For level $U_0(p^2)$ over $F$ we have $U_p=\sum_{s=1}^4 \mid {g_s}$ where \begin{equation}g_1:= \left(\begin{matrix} 2a_1 &0\\ 0 & d_1 \end{matrix} \right), g_2:=\left(\begin{matrix} 2a_2  & b \\ 2^2 c & d_2 \end{matrix} \right),g_3:=\left(\begin{matrix} 2 a_2 \zeta  & b \\ 2^2c & d_2 \zeta^2 \end{matrix} \right), g_4:= \left(\begin{matrix} 2 a_2 \zeta^2   & b \\ 2^2 c & d_2\zeta \end{matrix} \right)\end{equation} with
\begin{itemize}
	\item  $a_2,d_2 \in \ZZ_{p}$.
	\item  $a_2d_2=-bc=1/3$.
	\item $a_1=(2\zeta+1)a_2$.
	\item $d_1=-(2\zeta+1)d_2$.		
\end{itemize}
(We have not given an explicit description of the $a_i,d_i,b,c$ since this is cumbersome and we will only need the stated properties.)	

\end{prop}

\begin{proof}
This is a simple (computer) calculation. For the code, see \cite{githubme}.
\end{proof}

\section{Slopes}
In order for the space of modular forms of weight $[n_1,n_2]\chi$ to be  non-trivial one requires $\chi(x)=\Norm(x)^r$ for all $x \in \OO_F^\times$ where $(k,r,n,v,w)$ is the associated weight tuple. With this in mind we note that, since $p=2$, there are no non-trivial characters of $\OO_F$ with conductor $4$ such that $\chi(x)=\Norm(x)^2$ for all $x \in \OO_F^\times$ and there is a unique character with $\chi(x)=\Norm(x)$. We now set-up some notations that will be used throughout.
\begin{nota}
	
	\begin{itemize}
		\item  Our arithmetic weights will be of the form $[n_1,n_2]\chi$ where $n_1,n_2$ are both odd and $\chi$ is the unique non-trivial of $\OO_F$ with conductor $4$ on $\OO_F$ which sends the fundamental unit to $-1$ and sends $-1$ to $1$.
	
		\item 
		Throughout we will have the notational issue that the matrix representing the $U_p$ operator will be in terms of a basis of monomials $X^iY^j$ as this is the natural basis of $K \langle X,Y \rangle$. Moreover, our main technical results  will be about understanding precisely the $p$-adic valuations of each entry of the $U_p$ operator matrix, which will depend on the monomials to which the entry corresponds. For this reason we will speak of the $(\un{i},\un{j})$-th entry of a matrix, with $\un{i}=(i_1,i_2)$, where $\un{i}$ corresponds to the basis element $X^{i_1}Y^{i_2}$ and similarly for $\un{j}$.


		\item For $\un{i},\un{j}$ we let $\delta_{\un{i},\un{j}}=1$ if $i_1=j_1$ and $i_2=j_2$ and is $0$  otherwise.
		\item Let \[ \Delta_0(p^s) =\left \{ \g=\matrd{a}{b}{c}{d} \in M_2(\OO_K) :  p^s|c, p \nmid d, p|a, \det(\g) \neq 0  \right \}.\]
		\item Lastly, let $$C(i,j,n,x):=\sum_{r=0}^i \binom{n-j}{r} \binom{j}{i-r} x^r.$$  

	\end{itemize}
	
\end{nota}

	\begin{prop}\label{cor714}
	
	Let$$\g= \left(\begin{matrix} a&b\\ p^{s+1}c&d \end{matrix} \right) \in \Delta_0(p^{s+1})$$  and let  $\k=[n_1,n_2]\chi$ be an arithmetic weight with $(k,r,n,v,w)$ the associated weight tuple. Then the $\un{i},\un{j}$ entry of the matrix representing the $\mid_{\k} \g$ action on $K \langle X,Y	 \rangle$  is given by $$ \det(\g)^{v_1} \det(\ov{\g})^{v_2} \cdot  \Omega_{n}(\g,\un{i},\un{j})$$ where \begin{equation}
	\Omega_{n}(\g,\un{i},\un{j}):= \chi(d) \cdot d^{n_1}\ov{d}^{n_2}p^{i_1+i_2} \frac{a^{i_1}}{d^{j_1}}  \frac{\ov{a}^{i_2}}{\ov{d}^{j_2}} b^{j_1-i_1} \ov{b}^{j_2-i_2} C(i_1,j_1,n_1,\a) \cdot C(i_2,j_2,n_2,\ov{\a}),
	\end{equation}
 and $\a:=\frac{bc p^{s}}{ad}$.

\begin{proof}
	This is \cite[Corollary 3.1.17]{ME}.
\end{proof}\end{prop}

\begin{cor}\label{cor2}
	In weight $[n_1,n_2]\chi$,  a basis of\/ $K \langle X,Y \rangle^{\Gamma^{t}}$ is given by $X^{i_1}Y^{i_2}$ with $n_2-n_1 \equiv i_1-i_2 \bmod 3$.
\end{cor}
\begin{proof}
	This follows from Propositions \ref{prop26} and \ref{cor714}.\qedhere
\end{proof}

Now, we will normalize our $U_p$ operator (as in \cite{hidaon} ) by removing the factors $\det(\g)^{v_1} \det(\ov{\g})^{v_2}$ appearing in Corollary \ref{cor714}. We denote this operator by $U_p^0$.

\begin{cor}\label{cor6}
	The valuation of the $(\un{i},\un{j})$-th entry of $U_p^0$ is at least $i_1+i_2+g(i_1,j_1,n_1)+g(i_2,j_2,n_2)$ where $g(x,y,n)= \infty$ if  $x > n \geq y$, otherwise \begin{equation}
g(x,y,n) = \begin{cases} 
x & \text{if\/ } y=0, \\
0 & \text{if\/ } y \geq x, \\
x-y & \text{if\/ } y < x.
	\end{cases}
		\end{equation}
\end{cor}
\begin{proof}
	This is \cite[Corollary 3.1.18]{ME}
\end{proof}

\begin{prop}\label{ent}
 Let $U_2^0(\un{i},\un{j},n_1,n_2)$ denote the  $(\un{i},\un{j})$-th entry of\/ $U_p^0$ acting in weight $[n_1,n_2]\chi$, then \[U_p^0(\un{i},\un{j},n_1,n_2)= \sum_{s=1}^{4} \Omega_{n}(g_s,\un{i},\un{j})\] which, under a suitable change of basis, simplifies to \begin{equation}U_p^0(\un{i},\un{j},n_1,n_2):=E \cdot \left( C(i_1,j_1,n_1,-2)C(i_2,j_2,n_2,-2)+ 3^{m_1+m_2} \varepsilon_{\un{i},\un{j}}                     \right)\end{equation}
 where:
 \begin{itemize}
 	\item $E=2^{i_1+i_2}d_2^{n_1+n_2-2i_1-2i_2}3^{1-i_1-i_2}$.
 	\item   $\varepsilon_{\un{i},\un{j}} =(-1)^{m_1+m_2+i_1+i_2+1} \delta_{\un{i},\un{j}} $.
 	\item $d_2$ is as in Proposition \ref{prop1}.
 	\item $n_i=2m_i+1$ (which is true since we are only allowed weights whose algebraic part is odd). 
 \end{itemize}
 
\end{prop}

\begin{proof}
	Throughout the proof we are using the notation as in Proposition \ref{prop1}. The first part is immediate from Corollary \ref{cor714} and the fact that $U_p:= \sum_{s=1}^4 \mid_{g_s}$. We now need to prove the claimed simplification. 
	
	We begin by noting that since $g_1$ is diagonal, it will only contribute to $U_p^0(\un{i},\un{i},n_1,n_2)$ and it is easy to see from \ref{prop1} and \ref{cor714} that $\Omega_{n}(g_1,i,i):=\chi(d_1)2^{i_1+i_2}d_2^{n_1+n_2-2i_1-2i_2}(2\zeta+1)^{n_1-2i_2}(2\zeta^2+1)^{n_2-2i_1}$. Finally, one checks that $\chi(d_1)=-1$ and $(2\zeta+1)^2=-3$, which combines to give \[\Omega_{n}(g_1,\un{i},\un{j})=2^{i_1+i_2}d_2^{n_1+n_2-2i_1-2i_2}3^{m_1+m_2+1-i_1-i_2}(-1)^{m_1+m_2+1+i_1+i_2}\delta_{\un{i},\un{j}}.\]
	
	We now move on to the other three terms. Again, using Proposition \ref{prop1}, one checks easily that $\sum_{s=2}^4 \Omega_{n}(g_s,i,j)$ reduces to 
	 \[2^{i_1+i_2}d_2^{n_1+n_2}C(i_1,j_1,n_1,-2)C(i_2,j_2,n_2,-2)b^{j_1-i_1}\ov{b}^{j_2-i_2}3 \left ( \frac{a_2^{i_1+i_2}}{d_2^{j_1+j_2}}\right )\] which is the same as \[2^{i_1+i_2}d_2^{n_1+n_2-i_1-i_2-j_1-j_2}C(i_1,j_1,n_1,-2)C(i_2,j_2,n_2,-2)b^{j_1-i_1}\ov{b}^{j_2-i_2}3^{1-i_1-i_2}.\] The result now follows by conjugating by the diagonal matrix whose $(\un{i},\un{i})$-entry is $(bd_2)^{i_1}\ov{bd_2}^{i_2}$.\qedhere
	
\end{proof}

We now choose a specific ordering for our basis of $U_p^0$ acting in weight $\k:=[n_1,n_2]\chi$. Obviously the choice will not alter the result, but some choices will make it easier to prove what the slopes are. Specifically, we want to order the basis in such a way that the block diagonal submatrix of $U_p^0$ with $2 \times 2$-blocks has the same Newton polygon as $U_p^0$ (as we will prove below). We begin by first ordering the basis elements using the graded lexicographic order. Let $X^{a_1}Y^{b_1}$ be the first basis element under this ordering. We then set  $B_1(\k):=\{X^{a_1}Y^{b_1},X^{a_{1}+1}Y^{b_{1}+1}\}$ and  for $n \geq 1$, define \begin{equation}
B_{n+1}(\k):=B_{n}(\k) \cup \{X^{a_{n+1}}Y^{b_{n+1}},X^{a_{n+1}+1}Y^{b_{n+1}+1}\},
\end{equation} where $X^{a_{n+1}}Y^{b_{n+1}}$ the next basis element not already contained in $B_n$. Lastly, we set $B_{\infty}(\k)=\bigcup_n B_n(\k)$. 

\begin{nota}
	Let $U$ be an infinite matrix with respect to the basis $B_\infty(\k)$. Let $D(U)$ denote the  block diagonal submatrix of $U$ with blocks of size $2$. Note that each matrix along the diagonal will correspond to a pair of basis elements, say, $\{X^aY^b, X^{a+1}Y^{b+1}\}$. So we will denote these matrices by $D_{a,b}(U)$.
\end{nota}
\begin{rmrk}
Note that due to how our basis is ordered, such $a,b$ will either have different parity or both be even.
\end{rmrk}

\begin{prop}\label{prop8}
	\begin{enumerate}
		\item 	If\/ $a \not \equiv  b \bmod 2$ then the slopes of $D_{a,b}(U_p^0)$ are $a+b+1$ and $a+b+3$.
		\item 	If\/ $a \equiv b \equiv 0 \bmod 2$ then the slopes of $D_{a,b}(U_p^0)$ are $a+b+2$.
	\end{enumerate}

\end{prop}

\begin{proof}

	Write \begin{equation}
	D_{a,b}(U_p^0)= \begin{pmatrix}
	t_{a,b} & r_{a,b}\\ s_{a,b} & t_{a+1,b+1}
	\end{pmatrix}.	\end{equation} Then it follows from Proposition \ref{ent} that \begin{equation}
t_{a,b}=2^{a+b}u_{a,b}\left ( C(a,a,n_1,-2)C(b,b,n_2,-2)+(-3)^{m_1+m_2}(-1)^{a+b+1}\right )
	\end{equation} with $u_{a,b}$ a unit and $n_i=2m_i+1$. 
	
	Similarly, we have 
		\begin{align}
	r_{a,b}=&2^{a+b}u_{a,b}C(a,a+1,n_1,-2)C(b,b+1,n_2,-2) \\ s_{a,b}=&2^{a+b+2}u_{a+1,b+1}C(a+1,a,n_1,-2)C(b+1,b,n_2,-2)
	\end{align}
	
	Now a simple calculation shows that 
	\begin{align}
	&\val_p(t_{a,b}) =\begin{cases}
	a+b+1 & \text{ if } a \not \equiv b \bmod 2 \\
	a+b+2 + e & \text { if } a \equiv b \bmod 2
	\end{cases}\\
	&\val_p(r_{a,b}) =\begin{cases}
	a+b+2+e & \text{ if } a \not \equiv b \bmod 2 \\
	a+b & \text { if }  a \equiv b \equiv 0 \bmod 2
	\end{cases}\\	
	&\val_p(s_{a,b}) =\begin{cases}
	a+b+5+e & \text{ if }a \not \equiv b \bmod 2 \\
	a+b+4 & \text { if }  a \equiv b \equiv 0 \bmod 2
	\end{cases}	
		\end{align}
	where $e \ge 0$ is an error term. From this, the result follows at once.
	
\end{proof}

\begin{rmrk}
	Note that if $a,b$ are not both even, then the slopes of $D_{a,b}$ are determined by the diagonal entries and for $a \equiv b \equiv 0 \bmod 2$ the slopes are determined by the anti-diagonal entries.
\end{rmrk}

Let us now give an example which will motivate the proof of the main theorem.

\begin{exmp}\label{ex}
We begin by describing the basis in the case $\k=[1,1]\chi$ (recall that this is our notation for the weight which would usually be called parallel weight $3$ with nebentypus $\chi$). From Corollary \ref{cor2} we see that, for this weight, the basis for the space of overconvergent forms is given by $\mathcal{B}:=\{ X^iY^j : i \equiv j \bmod 3 \}$. Now, with the grlex ordering the first basis element is $1$, then $B_{1}(\k)=\{1,XY\}$. From this we get $B_{2}(\k)=B_1(\k) \cup \{X^3,X^4Y\}$, $B_{3}(\k)=B_2(\k) \cup \{Y^3, XY^4\}$ and so on.\footnote{From this it is clear that $B_{\infty}(\k)=\mathcal{B}$.}

Now, for a matrix $A$ with basis $B_\infty(\k)$, let $A(N)$ denote the truncation of $A$ to the top $N \times N$ left-hand corner of $A$ and let $D:=D(U_p^0)$. A computation shows that the slopes of $U_p^0(10)$ and $D(10)$ are both $[2,2],[4,2],[6,4],[8,2]$ where the first entry denotes the slope and the second its multiplicity. To explain this, we consider the matrix $V_p$ whose $(\un{i},\un{j})$-th entry is the $p$-adic valuation of the $(\un{i},\un{j})$-th entry of $U_p^0$. Using Proposition \ref{ent}, we see that $V_p(U_p^0(10))$ is 

\[ \begin{matrix}
*&  0&  0&  0&  0&  0&  0&  0&  0&  0& \\
4 & * & 3&  3 & 3&  3 & 6 & 2 & 7 & 2 &\\
* & * &  4 & 6 &* &* & 4&  3 & 4 & 3 &\\
* & * & 9 & 6 &* &* & 9 & 8 & 6 & 5 & \\
* &*& * &*  &4 & 6  &4  &3& * &* & \\
*& * &* &* & 9 & 6&  9 & 8 &*& *& \\
*&*&*&*&*&*  &8  &4& *& *& \\
*&*&*&*&*&* & 8 &11 &* &* &\\
*&* &10 & 9 &* &*& 10 &10 &12  &6 &\\
* &* &13 &12& * &* &15 &12 &10 &11 &  \\

\end{matrix}   \]

where here $*$ denotes an entry with infinite $p$-adic valuation, i.e. a zero entry of $U_p^0$. The above computation suggests that it is enough to understand the slopes of the matrices lying on the block diagonal. Looking at the above matrix one sees that the slopes are completely determined by the valuations of the entries. In fact, we will show that for $U_p^0$ the valuations of the entries completely determine the slopes, which in general is not true.
\end{exmp}


This example motivates the following definition:

\begin{defn}
	Let $\mathfrak{M}_\k$ denote the set of infinite matrices with respect to the basis $B_{\infty}(\k)$, such that if $A \in \mathfrak{M}_\k$ then the entries of $D_{a,b}(A)$ have valuations as given in the proof of Proposition \ref{prop8} and if $a_{\un{i},\un{j}}$ does not lie on $D(A)$ then $\val_p(a_{\un{i},\un{j}}) \ge i_1+i_2+g(i_1,j_1,n_1)+g(i_2,j_2,n_2)$ (as in Corollary \ref{cor6}).
\end{defn}

\begin{rmrk}
	Note that if $A \in \mathfrak{M}_\k$ then $D(A) \in \mathfrak{M}_\k$.
	
\end{rmrk}	

The following is the key result from which our main theorem will follow:

\begin{lem}\label{key lemma}
	All matrices in $\mathfrak{M}_\k$ have the same Newton polygon.
\end{lem}

We will defer the proof of this lemma to the next section and instead show how one can use it to deduce the main result.

\begin{cor}\label{cor224}
	The Newton polygons of\/ $U_p^0$ and $D(U_p^0)$ are the same.
\end{cor}

\begin{proof}
	By definition of $\mathfrak{M}_\k$ we have $U_p^0 \in \mathfrak{M}_\k$ from which the result follows.
\end{proof}

Let $\tau$ be the character  $\OO_K^\times \to \OO_{K,\tors}$ after fixing an isomorphism $\OO_K^\times \cong \OO_{K,\tors} \times \ZZ_p^2$. Note that $\OO_{K,\tors}$ is cyclic of order $6$. A simple computation gives the following table.
\begin{center}
	
	\begin{tabu}{| [2 pt] c | [1 pt]c |[2 pt]}

		\Xhline{2 pt}
		Weight &  [Slope, Multiplicity] at level $U_{0}(4)$ \\ \Xhline{2 pt}
		$[1,1]\chi$ & $[2,2]$
		\\ \hline
		$[1,1]\chi \tau^3$ & $[1,1],[3,1]$
		\\ \Xhline{2 pt} 	
		\multicolumn{2}{c}{Table 1: Classical slopes in different components of weight space.}
	\end{tabu}	
\end{center}
\begin{rmrk}
	We note that in order for the space of Hilbert modular forms to be non-trivial we require $\chi(e)\tau^n(e)=-1$ where $e$ is the fundamental unit in $\OO_F$. From this one checks that the only valid exponent of $\tau$ is $n \equiv 0 \bmod 3$ as $\tau(e)=\zeta^2$ (recall $\zeta$ is our fixed cube root on unity).
\end{rmrk}

\begin{nota}
	\begin{enumerate}
		\item Let $\S^\dagger([n_1,n_2]\chi)$ denote the multiset of slopes of $U_p^0$ acting on $S_{\k}^\dagger(U_0(p^2))$ with $\k=[n_1,n_2]\chi$. Similarly, let $\S([n_1,n_2]\chi)$ denote the set of classical slopes in weight $\k$ and level $U_0(p^2)$.
		\item For $r \in \ZZ$ and $\S([n_1,n_2]\chi)=\{s_1,s_2,\dots\}$ let $\S([n_1,n_2]\chi)+r=\{s_1+r,s_2+r, \dots\}$.
		\item For $a,b \in \ZZ$, let \[\chi_{a,b}:= \begin{cases}
		\chi & \text{if } a \equiv b \equiv 0 \bmod 2 \\
		\chi \tau^3 & \text{otherwise.}
		\end{cases}\]
		\item For $n \geq 1$ we let  $T_n(\k):=B_n(\k) \backslash B_{n-1}(\k)$, where $B_0(\k):=\emptyset$ and note that, \[T_n(\k):=\{X^{a_n(\k)}Y^{b_n(\k)},X^{a_n(\k)+1}Y^{b_n(\k)+1}\}\] for some $a_n(\k),b_n(\k) \in \ZZ_{\geq 0}$. Let $I_\k:=\{(a_n(\k),b_n(\k))  \}_n$.
	\end{enumerate}
\end{nota}

\begin{thm}\label{thm:main}
	Let $n_i \in \ZZ$ with $n_i$ odd and  let $\k=[n_1,n_2]\chi$. Then \[\S^\dagger([n_1,n_2]\chi) =\bigcup_{(a,b) \in I_\k} \S([1,1]\chi_{a,b}) +a+b.  \]
\end{thm}

\begin{proof}
	This follows from Corollary \ref{cor224} combined with Proposition \ref{prop8} and Table $1$.
\end{proof}

	Note that twisting our weight by $\tau$ has the effect of changing the component of weight space that we are in. The above theorem then says that once we know a small set of classical slopes in each component of weight space, then we can obtain all slopes in any component of weight space.

\begin{rmrk}\label{rmrk:HN} \begin{enumerate}
		\item 

	The condition that $n_i$ is odd is required in order for the relevant space of overconvergent Hilbert modular forms to be non-trivial.

\item 	Note that this is a simple generalization of \cite[Conjecture 4.7.1 and Conjecture 4.7.8]{ME} to levels which are not \textit{sufficiently small} (meaning we have $\Gamma^t/ \Gamma^t \cap F^\times$ is non-trivial). Having a level which is not sufficiently small has the effect of making the set over which we index in Theorem \ref{thm:main} more complicated.

\item Computations suggest that, in this case, the Newton and Hodge polygons associated to $U_p$ never touch (after the first vertex where they touch for trivial reasons) and therefore the methods of \cite{eovb} cannot be used to describe the slopes. Specifically, if one computes the Newton and Hodge polygons of $U_p(N)$ for $N >>0$ then the polygons only touch at the endpoints (which will always be the case by the definition of the Newton and Hodge polygon associated to a finite matrix).
		\end{enumerate}
\end{rmrk}

\section{Proof of  Lemma \ref{key lemma}}

As mentioned above, the proof relies on the coincidence that all matrices in $\mathfrak{M}_\k$ have the same slopes. The strategy of proof relies on showing that the entries on the block diagonal submatrix $D(U_p^0)$  determine the $p$-adic valuation of a coefficient of the characteristic power series. To show this we will first write down an explicit formula for the coefficients of the characteristic power series in terms of the principal minors and then use our knowledge of valuations of the entries to determine when the valuation of the principal minor is minimized.

\begin{nota}
	\begin{enumerate}
		\item 	Let $I$ denote the indexing set for the elements of $B_{\infty}(\k)$. We will continue to denote these indexes by $\un{i}$. Furthermore, for $n \in \ZZ_{\geq 1}$ let  $I_n$ denote the subsets of size $n$ of $I$.
		\item For $J \subset I$ let $S(J):=\sum_{\un{i} \in J} i_1+i_2$.
		\item If $A$ is an infinite matrix with basis given by $B_\infty(\k)$ and $J \in I_n$ we let $[A]_J$ denote the principal minor defined by $J$.
	\end{enumerate}

\end{nota}		
	
One can compute the coefficients of the characteristic polynomial of $U_p^0$ using the following result of Serre.

\begin{prop}\label{serr}
	Let $U=(u_{\un{i},\un{j}})$ be the matrix associated to a compact operator (after choosing a basis). Then $\det(1-TU)= \sum_{n=0}^{\infty} c_n(U) T^n$ where \begin{equation}
	c_n(U)=(-1)^n \sum_{ \substack{J \subset I_n}} \sum_{\s \in \Sym(J)} sgn(\s) \prod_{\un{i} \in J}u_{\s(\un{i}),\un{i}}=(-1)^n \sum_{ \substack{J \subset I_n}} [U]_J
	\end{equation} where $I$ is the indexing set of the basis elements.
\end{prop} 

\begin{proof}
	This is \cite[Proposition 7]{serro}.
\end{proof}

Now, we want to use Proposition \ref{prop8} to find the minors of $D(U_p^0)$ which will have valuation as small as possible.  Since $D(U_p^0)$ is a block diagonal matrix, its slopes are determined by the slopes of the blocks.

If $\{\lambda_i\}$ denotes the multiset of slopes appearing in $D(U_p^0)$, with $\lambda_i \leq \lambda_{i+1}$, then the Newton polygon of $D(U_p^0)$ is formed by the points $(n, \sum_{i=1}^n \lambda_i)$ and each $\lambda_i$ will come from one of the slopes appearing in $D_{a_{\lambda_i},b_{\lambda_i}}(U_p^0)$ for some index $(a_{\lambda_i},b_{\lambda_i}) \in B_{\infty}(\k)$. Moreover, the breakpoints of the Newton polygon occur when we have $\lambda_n < \lambda_{n+1}$. 

By Proposition \ref{prop8}, we see that if $a_{\lambda_i},b_{\lambda_i}$ have different parity, then \[\lambda_i \in \{\val_p(t_{a_{\lambda_i},b_{\lambda_i}}),\val_p(t_{a_{\lambda_i}+1,b_{\lambda_i}+1})\}.\] Otherwise, $D_{a_{\lambda_i},b_{\lambda_i}}(U_p^0)$ has the same slopes $\lambda_i$.

\begin{prop}\label{prop9}
	Let $\lambda_i=\lambda_j$ with $a_{\lambda_i},b_{\lambda_i}$ having different parity and $a_{\lambda_j},b_{\lambda_j}$ having same parity, then \[\val_p(t_{a_{\lambda_i},b_{\lambda_i}}) \leq \val_p(t_{a_{\lambda_j},b_{\lambda_j}}).\] 

\end{prop}
\begin{proof}
	First note that $\lambda_i=\val_p(t_{a_{\lambda_i},b_{\lambda_i}})= a_{\lambda_i}+b_{\lambda_i}+1$ or $\lambda_i=\val_p(t_{a_{\lambda_i}+1,b_{\lambda_i}+1})=a_{\lambda_i}+b_{\lambda_i}+3$. Now by Proposition \ref{prop8}, $\val_p(t_{a_{\lambda_j},b_{\lambda_j}})=\lambda_j + \epsilon$ with $\epsilon \geq 0$ from which the result follows.
\end{proof}

Let us now reorder the multiset of $\{\lambda_i\}$, so that if $\lambda_i=\lambda_j$ where $a_{\lambda_i},b_{\lambda_i}$ have different parity and $a_{\lambda_j},b_{\lambda_j}$ have same parity, then we relabel the slopes so that $i \leq j$. Similarly, we relabel the same slopes of $D_{a_{\lambda_i},b_{\lambda_i}}(U_p^0)$ by $\lambda_i$ and $\lambda_{i+1}$.

\begin{lem}\label{lem1}
For all $J \in I_n$ we have $\val_p([D(U_p^0)]_J) \geq \sum_{i=1}^n \lambda_i$. Moreover, if $\lambda_n < \lambda_{n+1}$ then there exists a unique $J$ such that $\val_p([D(U_p^0)]_J) = \sum_{i=1}^n \lambda_i$.
\end{lem}

\begin{proof}
	Let $D:=D(U_p^0)$. Note that from the definition of a Newton polygon, we have \[\sum_{i=1}^n \lambda_i \leq \val_p(C_n(D))= \val_p \left ( \sum_{J \in I_n} [D]_J \right ).\] Now, let $J_n=\{(a_{\lambda_i},b_{\lambda_i} )\}_{i=1}^n$, then by Propositions \ref{prop8} and \ref{prop9}, it follows that $\val_p([D]_{J_n}) \leq \val_p([D]_J)$ for all $J \in I_n$. In fact, if $a_{\lambda_n},b_{\lambda_n}$ have different parity, then $\val_p([D(U_p^0)]_{J_n}) = \sum_{i=1}^n \lambda_i$. 
	
	Now, note that if we have $\lambda_n=\lambda_{n+1}$ with $a_{\lambda_n},b_{\lambda_n}$ having different parity then $J_{n-1} \cup \{a_{\lambda_n},b_{\lambda_n}\}$ and $J_{n-1} \cup \{a_{\lambda_{n+1}},b_{\lambda_{n+1}} \}$ give rise to principal minors with the same $p$-adic valuation. In general, one sees that the index set $J$ such that $[U]_J$ is minimal are constructed by rearranging the $\lambda_i,\lambda_j$ (and thus the corresponding indexes) with $\lambda_i=\lambda_j$. From this it is then clear that if $\lambda_n < \lambda_{n+1}$, then $J_n$ is the unique subset in $I_n$ for which $\val_p([D(U_p^0)]_{J_n}) = \sum_{i=1}^n \lambda_i$, which finishes the proof.\qedhere

\end{proof}
\begin{cor}\label{kcor}
	For each $n$ the minimal valuation of a $n \times n$ principal minor of $D(U_p^0)$ lies on or above $\NP(D(U_p^0))$. Furthermore, if\/ $(m,\val_p(c_{m}(D(U_p^0)))$ is a vertex, then there is a unique minor $[D(U_p^0)]_J$ with $J \in I_m$ such that \[\val_p([D(U_p^0)]_J)=\val_p(c_{m}(D(U_p^0)).\] 
\end{cor}

\begin{nota}
Let $\JJ_{n}$  denote the set of $J \in I_{n}$ such that $\val_p([D(U_p^0)]_J$ is minimal.
\end{nota}
 For $J \in \JJ_n$ we want to select  $\s_J \in \Sym(J)$ which will pick out the entries with smallest valuation. From Proposition \ref{prop8}, its clear that to do this we must take the diagonal entries of any block $D_{a,b}(U_p^0)$ with $a \not \equiv b \bmod 2$ and the anti-diagonal entries if $a \equiv b \equiv 0 \bmod 2$. The only issue is that if the last index, $(a,b)$ say,  in $J$ is such that $a \equiv b \equiv 0 \bmod 2$, then we cannot pick the anti-diagonal entries, in this case one must again take the diagonal entry.

With this in mind, we make the following definition:
\begin{defn}
 For $J \in \JJ_n$, we take $\s_J \in \Sym(J)$  such that  if $i_1 \not \equiv i_2 \bmod 2$ then $\s_J(\un{i})=\un{i}$,  $\s_J(\un{j})=\un{j}$  and otherwise we take  $\s_J(\un{i})=\un{j}$ where   $j_1=i_1+1, j_2=i_2+1$, except possibly for the final index in $J$. Note that Proposition \ref{prop8} and Lemma \ref{lem1} ensure that $\s_J$ is chosen so as to pick out the entries lying on $D(U_p^0)$ with smallest valuation. 

\end{defn}

With this we can now prove the lemma.

\begin{proof}[Proof of Lemma \ref{key lemma}]
	By Proposition \ref{prop8} it is clear that if $A,B \in \mathfrak{M}_\k$ then $\NP(D(A))=\NP(D(B)).$  So it is enough to check that for any $A=(a_{\un{i},\un{j}}) \in \mathfrak{M}_\k$ we have $\NP(A)=\NP(D(A))$.

	Now, by Propositions \ref{serr} we have  \begin{equation}
	\val_p(c_n(A)) \ge \min_{ \substack{J \subset I_n \\ \s \in \Sym(J)}}\left \{  \sum_{\un{i} \in J} \val_p(a_{\s(\un{i}),\un{i}}) \right \}
	\end{equation}

and, for fixed $J,\s$, we have 
		
\begin{equation}
 \sum_{\un{i} \in J} \val_p(a_{\s(\un{i}),\un{i}}) =  S(J)+ \left( \sum_{\un{i} \in J} \val_p(a'_{\s(\un{i}),\un{i}}) \right )
\end{equation} where $a'_{\un{i},\un{j}}=2^{-i_1-i_2}a_{\un{i},\un{j}}$ which is in $\OO_K$ by Corollary \ref{cor6}. 
	
	Now, note that:
	
	\begin{enumerate}
		\item 	Since our basis is given by elements $X^{i_1}Y^{i_2}$ with $n_2-n_1 \equiv i_1-i_2 \bmod 3$, it follows that, if $a_{\un{i},\un{j}}$ is an entry of $A$,   then $|i_1-j_1|+|i_2-j_2| \ge 2$, with equality if and only if $a_{\un{i},\un{j}}$ lies in $D(A)$ (this is due to how we have ordered our basis). Furthermore, by Corollary \ref{cor6} it is easy to check that  if $|\s(\un{i})_1-i_1|+|\s(\un{i})_2-i_2| > 2$ then  $$\val_p(a'_{\s(\un{i}),\un{i}}a'_{\un{i},\s(\un{i})}) > 2.$$ In other words, if the entries are `far' from the block diagonal then their product has large valuation.
		
		\item If $i,j \in J$ with $j_1=i_1+1,j_2=i_2+1$ then by Proposition \ref{prop8}, there exist $\s$ such that   $$\val_p(a'_{\s(\un{i}),\un{i}}a'_{\un{i},\s(\un{i})}) = 2.$$ Explicitly, we take $\s$ such that $a'_{\s(\un{i}),\un{i}},a'_{\un{i},\s(\un{i})}$ are both on the diagonal or anti-diagonal of $D_{i_1,i_2}(A)$ (depending on the parity of $(i_1,i_2)$).
		
		\item  If $\un{i} \in J$ with $i_1 \not \equiv i_2 \bmod 2$ then $\val_p(a'_{\un{i},\un{i}})=1$ and otherwise  $\val_p(a'_{\un{i},\un{i}}) \ge 2$.
	\end{enumerate}

 From the above it follows that the $J \in I_n$ and $\s \in \Sym(J)$ for which$$\sum_{\un{i} \in J} \val_p(a_{\s(\un{i}),\un{i}})$$ is minimal are those with $J \in \JJ_n$ and $\s=\s_J$. This also shows that the minimal $p$-adic valuation of the $n \times n$ principal minors of $A$ and $D(A)$ coincide. Now, note that there could be several such $J$ with $\val_p([A]_J)$ being minimal, so we only know that $(n,\val_{p}(c_n(A)))$ lies on or above $\NP(D(A))$ (this follows from Corollary \ref{kcor}). To get the result, we note that by Corollary \ref{kcor} if $n$ corresponds to a vertex, then there is a unique $J$ such that the minor $[D(A)]_J$ has minimal valuation and the above shows that the same is true for $A$, therefore $(n, \val_p(c_{n}(A)))$ is a vertex of $\NP(A)$.\qedhere

\end{proof}

\section{Scaling of slopes}
We now have the following computational evidence which shows that as our weights approach the boundary of weight space, the slopes scale.

Let $\chi=\chi_{2}$ denote our fixed finite character as in the previous section. Then  for $n \geq 3$, let $\chi_{n}$ be a character of conductor $p^n$, such that $\chi_{n}(e)=-1$ and $ \chi(-1)=1$ for $e$ the fundamental unit in $\OO_F$. Similarly, for $n \geq 3$ we let $\psi_{n}$ denote characters where $\psi_{n}(e)=1$ and $\psi(-1)=1$.

\vspace{5mm}
\begin{tabu}{| [2 pt] c | [1 pt]c |  [1 pt]X |  [2 pt]}
	\Xhline{2 pt}
	
	Weight & level &  [Slope, Multiplicity]  \\ \Xhline{2 pt}
	$[3,3]\chi_2$ & $U_0(4)$ & $[2,2]$ 
		\\ \hline
		$[3,3]\chi_3$ &  $U_0(8)$ & $[ 1, 2 ],[ 2, 2 ],[ 3, 2 ]$
		\\ \hline
			$[3,3]\chi_4$ & $U_0(16)$ &   $ [ 1/2, 2 ],
		[ 1, 2 ],
		[ 3/2, 4 ],
		[ 2, 6 ],
		[ 5/2, 4 ],
		[ 3, 2 ],
		[ 7/2, 2 ]$	
		\\ \hline	
			$[3,3]\chi_5$ & $U_0(32)$ & 
		$ [ 1/4, 2 ],
		[ 1/2, 2 ],
		[ 3/4, 4 ],
		[ 1, 6 ],
		[ 5/4, 6 ],
		[ 3/2, 8 ],
		[ 7/4, 10 ],
		[ 2, 10 ],
		[ 9/4, 10 ],
		[ 5/2, 8 ],
		[ 11/4, 6 ],\allowbreak
		[ 3, 6 ],
		[ 13/4, 4 ],
		[ 7/2, 2 ],
		[ 15/4, 2 ]	$
		\\ \Xhline{2 pt} 	  
	$[ 2, 2 ]\psi_3$& $U_0(8)$ &
$[ 1, 2 ]$
\\ \hline	
	$[ 2, 2 ]\psi_4$& $U_0(16)$ &
$[ 1/2, 2 ],
[ 1, 2 ],
[ 3/2, 2 ]$
\\ \hline	
	$[2,2]\psi_5$ &    $U_0(32)$ &        $[ 1/4, 2 ],
[ 1/2, 2 ],
[ 3/4, 4 ],
[ 1, 6 ],
[ 5/4, 4 ],
[ 3/2, 2 ],
[ 7/4, 2 ]$
			\\ \Xhline{2 pt} 	
			
	$[3,5]\chi_2$ & $U_0(4)$ & 	$[ 2, 1 ],
	[ 4, 1 ]$
\\ \hline
$[3,5]\chi_3$ &  $U_0(8)$ & 	$[ 1, 1 ],
[ 2, 3 ],
[ 3, 2 ],
[ 4, 3 ],
[ 5, 1 ]$
\\ \hline
$[3,5]\chi_4$ & $U_0(16)$ &   	$[ 1/2, 1 ],
[ 1, 3 ],\allowbreak
[ 3/2, 4 ],\allowbreak
[ 2, 5 ],\allowbreak
[ 5/2, 5 ],\allowbreak
[ 3, 6 ],\allowbreak
[ 7/2, 5 ],\allowbreak
[ 4, 5 ],
[ 9/2, 4 ],\allowbreak
[ 5, 3 ],\allowbreak
[ 11/2, 1 ]$
\\ \hline	
$[3,5]\chi_5$ & $U_0(32)$ &   $ [ 1/4, 1 ],
[ 1/2, 3 ],\allowbreak
[ 3/4, 4 ],\allowbreak
[ 1, 5 ],\allowbreak
[ 5/4, 7 ],\allowbreak
[ 3/2, 8 ],\allowbreak
[ 7/4, 9 ],\allowbreak
[ 2, 11 ],\allowbreak
[ 9/4, 10 ],\allowbreak
[ 5/2, 11 ],\allowbreak
[ 11/4, 11 ],\allowbreak
[ 3, 10 ],\allowbreak
[ 13/4, 11 ],\allowbreak
[ 7/2, 11 ],\allowbreak
[ 15/4, 10 ],\allowbreak
[ 4, 11 ],\allowbreak
[ 17/4, 9 ],\allowbreak
[ 9/2, 8 ],\allowbreak
[ 19/4, 7 ],\allowbreak
[5, 5],\allowbreak
[ 21/4, 4 ],\allowbreak
[ 11/2, 3 ],\allowbreak
[ 23/4, 1 ]$
\\ \Xhline{2 pt} 	
	$[ 2, 4 ]\psi_3$& $U_0(8)$ &
  $[ 1, 1 ],
[ 2, 2 ],
[ 3, 1 ]$
\\ \hline	
$[ 2, 4 ]\psi_4$& $U_0(16)$ &
  $[ 1/2, 1 ],
[ 1, 3 ],\allowbreak
[ 3/2, 3 ],\allowbreak
[ 2, 2 ],\allowbreak
[ 5/2, 3 ],\allowbreak
[ 3, 3 ],
[ 7/2, 1 ]$
\\ \hline	
$[2,4]\psi_5$ &    $U_0(32)$ &        $[ 1/4, 1 ],
[ 1/2, 3 ],\allowbreak
[ 3/4, 4 ],\allowbreak
[ 1, 5 ],\allowbreak
[ 5/4, 6 ],\allowbreak
[ 3/2, 5 ],\allowbreak
[ 7/4, 5 ],\allowbreak
[ 2, 6 ],\allowbreak
[ 9/4, 5 ],\allowbreak
[ 5/2, 5 ],\allowbreak
[ 11/4, 6 ],\allowbreak
[ 3, 5 ],\allowbreak
[ 13/4, 4 ],\allowbreak
[ 7/2, 3 ],\allowbreak
[ 15/4, 1 ]$
\\ \Xhline{2 pt}  			  	
\end{tabu}
\vspace{5mm}
Unfortunately proving the scaling behaviour seems to  be beyond our current methods, but the above tables suggest that as one approaches the boundary of weight space the slopes scale in a way analogous to what occurs for elliptic modular forms.

 \bibliographystyle{alpha}
 
 \bibliography{biblio(3)}

\end{document}